\documentclass{amsart}
\usepackage{amsthm, amssymb, amsfonts}
\usepackage{lmodern}
\usepackage{color}
\usepackage[T5]{fontenc}
\usepackage{amscd,amssymb}
\usepackage[v2,cmtip]{xy}
\usepackage{mathrsfs}
\theoremstyle{plain}
\newtheorem{thm}{Theorem}[section]
\newtheorem{lem}[thm]{Lemma}

\newtheorem{prop}[thm]{Proposition}
\theoremstyle{definition}
\newtheorem{defn}[thm]{Definition}

\def\vdvh{\vrule height 12pt depth 6pt width 0.5pt}
\def\vdq{\vdvh\ }

\begin{document}

 
\title[On the action of the Steenrod-Milnor operations]{On the action of the  Steenrod-Milnor operations\\ on the invariants of  the general linear groups }

 \author{Nguyen Thai Hoa, Pham Thi Kim Minh, Nguyen Sum}

\vskip2.5cm

\maketitle

\section{Introduction}

Let  $p$ be an odd prime number. Denote by $GL_n  =  GL(n,\mathbb F_p)$ the general linear  group over the prime field $\mathbb F_p$. Each subgroup of $GL_n$ acts on  the algebra $P_n=E(x_1,\ldots,x_n)\otimes \mathbb F_p(y_1,\ldots,y_n)$ in the usual manner. Here and in what follows, $E(.,\ldots,. )$ and $\mathbb F_p(.,\ldots,.)$ are the exterior and polynomial algebras over $\mathbb F_p$ generated by the indicated variables. We grade $P_n$ by assigning $\dim x_i=1$ and $\dim y_i=2.$

Dickson showed in \cite{1} that the invariant algebra $\mathbb F_p(y_1,\ldots,y_n)^{GL_n}$ is a polynomial algebra generated by the Dickson invariants $Q_{n,s},\ 0\le s<n$.  Hu\`ynh M\`ui \cite{2,3} computed the invariant algebras  $P_n^{GL_n}$. He proved that $P_n^{GL_n}$ is generated by  $Q_{n, s},\ 0 \le  s <  n,\ R_{n, s_1,  \ldots, s_k},\  0 \le s_1 < \ldots < s_k < n.$ Here $R_{n, s_1,  \ldots, s_k}$ are M\`ui invariants and $Q_{n,s}$ are  Dickson invariants (see Section 2). 

Let $\mathcal A_p$ be the mod $p$ Steenrod algebra and let
$\tau_s, \xi_i$ be the Milnor elements of dimensions $2p^s-1,\
2p^i-2$ respectively in the dual algebra $\mathcal A_p^*$ of $\mathcal A_p$.
In \cite{7}, Milnor showed that as an algebra,
$$\mathcal A_p^* = E(\tau_0,\tau_1,\ldots )\ \otimes\ \mathbb F_p(\xi_1,\xi_2,\ldots ). $$
Then $\mathcal A_p^*$ has a basis consisting of all monomials
    $$\tau_S\xi^R \ =\ \tau_{s_1}\ldots \tau_{s_k}\xi_1^{r_1}\ldots
     \xi_m^{r_m},$$
with $S = (s_1,\ldots ,s_k),\ 0 \le s_1 <\ldots <s_k ,
 R = (r_1,\ldots ,r_m),\ r_i \ge 0 $. Let $St^{S,R} \in \mathcal A_p$
denote the dual of $\tau_S\xi^R$ with respect to that basis.
Then $\mathcal A_p$ has a basis consisting of all operations $St^{S,R}$. For $S=\emptyset, R=(r)$, $St^{\emptyset, (r)}$ is nothing but the Steenrod operation $P^r$. So, we call $St^{S,R}$ the Steenrod-Milnor operation of type $(S,R)$.

We have the Cartan formula
$$St^{S,R}(uv) =\sum_{\overset{\scriptstyle{S_1\cup S_2=S}}{R_1+R_2=R}}St^{S_1,R_1}(u)St^{S_2,R_2}(v),$$
where $ R_1=(r_{1i}),\ R_2=(r_{2i}),\ R_1+R_2=(r_{1i}+r_{2i}), S_1\cap S_2=\emptyset, u,v\in P_n$ (see M\`ui \cite{3}). 

We denote $ St_u=St^{(u),(0)},\ St^{\Delta_i}=St^{\emptyset,\Delta_i}$, where $\Delta_i=(0,\ldots,$ $1,\ldots,0)$ with 1 at the $i$-th place.  In  \cite{3}, Hu\`ynh M\`ui proved that as a coalgebra, $$\mathcal A_p =\Lambda(St_0,St_1,\ldots)\otimes \Gamma(St^{\Delta_1},St^{\Delta_1},\ldots ).$$ Here, $\Lambda(St_0,St_1,\ldots)$ (resp. $\Gamma(St^{\Delta_1},St^{\Delta_2},\ldots )$) denotes the exterior (resp. polynomial) Hopf algebra with divided powers generated by the  primitive Steenrod-Milnor operations $St_0,\ \!St_1,\ldots$ (resp. $St^{\Delta_1}, St^{\Delta_2},\ldots )$. 

The Steenrod algebra  $\mathcal A_p$ acts on $P_n$ by means of the Cartan formula together with the relations 
\begin{align*} 
St^{S,R}x_k&=\begin{cases} x_k, & S=\emptyset,\ R=(0),\\
y_k^{p^u}, &S=(u),\ R=(0), \\ 
0,  &otherwise, 
\end{cases}\\
St^{S,R}y_k&=\begin{cases}
 y_k, & S=\emptyset ,\ R=(0),\\
y_k^{p^i},&S=\emptyset,\ R=\Delta_i, \\ 
0, &otherwise,  \end{cases}
\end{align*} 
for $k=1,\ \! 2, \ldots, n$ (see Steenrod-Epstein \cite{13}, Sum \cite{11}). 
Since this action commutes with  the action of $GL_n$, it induces  an actions of $\mathcal A_p$ on $P_n^{GL_n}$.

The action of $St^{S,R}$ on the modular invariants of subgroups of general linear group has partially been studied by many authors. This action for $S=\emptyset,\ \! R=(r)$ was explicitly determined by Smith-Switzer \cite{12},  Hung-Minh \cite{5}, Kechagias \cite{4}, Sum \cite{11,20,21,22}. Smith-Switzer \cite{12}, Wilkerson \cite{14} have studied the action of $St^{\Delta_i}$ on the Dickson invariants.

The purpose of the paper is to compute the action of the Steenrod-Milnor operations on the generators of $P_2^{GL_2}$. More precisely, we explicitly determine the action of $St^{(i,j)}$ on the Dickson invariants $Q_{2,0}$ and $Q_{2,1}$.

The analogous results for $p = 2$ have been anounced in \cite{9}.

The rest of the paper contains two sections. In Section 2, we recall some needed information on the invariant theory. In Section 3, we compute the action of  the Steenrod-Milnor operations on the Dickson invariants. 

\section{Preliminaries}

\begin{defn} 
Let $(e_{k+1 },\ldots,e_n),\ 0  \leq   k <  n$, be a sequence of non-negative integers.  Following Dickson \cite{1}, M\`ui \cite{2}, we define  
$$ [k;e_{k+1},  \ldots,  e_n]  =  \frac 1{k!}
\vmatrix x_1&\cdots &x_n\\
  \vdots&\cdots &\vdots\\
  x_1&\cdots &x_n\\
  y_1^{p^{e_{k+1}}}&\cdots &y_m^{p^{e_{k+1}}}\\
  \vdots&\cdots  &\vdots\\
   y_1^{p^{e_n}} & \cdots & y_m^{p^{e_n}}
   \endvmatrix, $$
 in which there are exactly $k$ rows of $(x_1 \ldots x_n)$. (See the accurate meaning of the
determinants in a commutative graded algebra in M\`ui n \cite{2}). For $k=0$, we write
$$[0;e_{1},  \ldots,  e_n] =  [e_{1},  \ldots,  e_n] =\det(y_i^{p^{e_j}}).$$
In particular, we set
 \begin{align*} L_{n,s}&=[0,1,  \ldots,\hat s,\ldots,  n],\ \! 0\le s  \le n,\\ 
L_n &=L_{n,n}=[0,1,\ldots,n-1].
\end{align*}
Each $[k;e_{k+1},  \ldots,  e_n]$ is an invariant of the special linear group $SL_n$ and $[e_{1},  \ldots,  e_n]$ is divisible by $L_n$. Then, Dickson invariants $Q_{n,s},\ 0 \le s < n$,  are defined by
$$ Q_{n,s}=L_{n,s}/L_n.$$
Here, by convention, $L_0=[\emptyset] =1.$
\end{defn}

Now we recall the following which will be used in the next section.

\begin{lem}[Sum \cite{20}]
The actions of  $St^R$  on $L_2,L_{2,0},L_{2,1}$ for
$\ell(R)=2$ are respectively  given in the following tables:

\medskip
{\begin{tabular}{cllcl}
 $(i,j)$ &\vdq $St^{(i,j)}L_2 $ &\vdq& $(i,j)$ &\vdq $St^{(i,j)}L_2 $\cr
\hline                                    
$(0,0)$ &\vdq $L_2 $ &\vdq& $(0,1) $&\vdq $-L_2Q_{2,0}$ \cr
$(0,p) $&\vdq $L_2(Q_{2,1}^{p+1}-Q_{2,0}^p)$ & \vdq& $(0,p+1)$ &\vdq $L_2Q_{2,0}^{p+1}$\cr
 $(1,p) $&\vdq $L_2Q_{2,0}Q_{2,1}^p$ &\vdq& $(p,0)$ &\vdq $L_2Q_{2,1}$\cr
$(p+1,0)$ &\vdq $L_2Q_{2,0}$ &\vdq&\text{otherwise}&\vdq 0,\cr
\end{tabular}} 

\medskip
{\begin{tabular}{cllcl}
 $(i,j)$ &\vdq $St^{(i,j)}L_{2,0} $ &\vdq& $(i,j)$ &\vdq $St^{(i,j)}L_{2,0} $\cr
\hline                                            
$(0,0)$ &\vdq $L_2Q_{2,0}$ &\vdq &$(0,p)$ &\vdq $-L_2Q_{2,0}^{p+1} $ \cr
$(0,p^2)$ &\vdq $L_2(Q_{2,0}Q_{2,1}^{p^2+p}-Q_{2,0}^{p^2+1})$ &\vdq& $(0,p^2+p)$ &\vdq $L_2Q_{2,0}^{p^2+p+1}$ \cr
$(p,p^2) $&\vdq $L_2Q_{2,0}^{p+1}Q_{2,1}^{p^2}$ &\vdq & $(p^2,0)$ &\vdq $L_2Q_{2,0}Q_{2,1}^p$ \cr
$(p^2+p,0)$ &\vdq $L_2Q_{2,0}^{p+1}$ &\vdq& \text{otherwise} &\vdq 0,\cr
\end{tabular}} 

\medskip
{\begin{tabular}{cl}
$(i,j)$ &\vdq $St^{(i,j)}L_{2,1}$ \cr
\hline                                             
$(0,0) $&\vdq $L_2Q_{2,1}$ \cr 
$(0,p^2)$ &\vdq $L_2(Q_{2,1}^{p^2+p+1}-Q_{2,0}^pQ_{2,1}^{p^2}-Q_{2,0}^{p^2}Q_{2,1})$\cr
$(0,p^2+1)$ &\vdq $L_2Q_{2,0}^{p+1}Q_{2,1}^{p^2}$ \cr 
$(1,0)$ &\vdq $L_2Q_{2,0}$\cr
$(1,p^2)$ &\vdq $L_2(Q_{2,0}Q_{2,1}^{p^2+p}-Q_{2,0}^{p^2+1})$ \cr 
$(p^2,0)$ &\vdq $L_2(Q_{2,1}^{p+1}-Q_{2,0}^p)$ \cr
$(p^2,1)$ &\vdq $L_2Q_{2,0}^{p+1}$ \cr 
$(p^2+1,0)$ &\vdq $L_2Q_{2,0}Q_{2,1}^{p}$\cr
\text{otherwise} &\vdq 0.\cr
\end{tabular}} 
\end{lem}

\eject
\section{Main Results}

First of all, we recall the following.
\begin{prop}[Hung-Minh \cite{5}, Sum  \cite {11}] Let $i$ be a nonnegative integer. We have 
\begin{align*} &St^{(i,0)}Q_{2,0}= \begin{cases}  Q_{2,0}^p,  &i=p^2-1,\\ 
(-1)^k\binom krQ_{2,0}^{r+1}Q_{2,1}^{k-r},& i = kp +r ,\ 0\le r \le k < p,\\
0, &\text{ otherwise},\end{cases}\\
 &St^{(i,0)}Q_{2,1}= \begin{cases}  Q_{2,1}^p,  &i=p^2-p,\\ 
(-1)^k\binom {k+1}rQ_{2,0}^{r}Q_{2,1}^{k+1-r},& i = kp +r,\ r \le k+1\le p,\\
0, &\text{ otherwise}.\end{cases}
\end{align*}
\end{prop}

\begin{thm}[Sum \cite{20}] Let $j$ be a nonnegative integer and $s=0,1$. Then we have
$$St^{(0,j)}Q_{2,s}= \begin{cases} 
Q_{2,0}^{r}Q_{2,s}\sum_{i=0}^k(-1)^i   \binom{k+s}{i +s}\binom
{r+i}iQ_{2,0}^{(k-i)p}Q_{2,1}^{i(p+1)},\\
\hskip2.2cm \text{ if } j=kp+r  \text{ with } 0\le  k,\ r <p,\\
0, \hskip2cm \text{otherwise}.\end{cases}$$
\end{thm}
Our new results are the actions of $St^{(i,j)}$ on $Q_{2,0}$ and $Q_{2,1}$ with $i>0.$ 
\begin{thm}\label{dl33} Let $i, j$ be nonnegative integers. We have 

{\rm(i)} If $0\le k<i<p$ and $0\le r <p$ then
$$St^{(i,kp+r)}Q_{2,0}=0.$$

{\rm (ii)} If $0\le i <p$ then
$$St^{(i,ip)}Q_{2,0}=(-1)^iQ_{2,0}^{i+1}Q_{2,1}^{ip}.$$

{\rm (iii)}
$$ St^{(1,j)}Q_{2,0}= \begin{cases} -Q_{2,0}^{r+2}Q_{2,1}^p\sum_{i=0}^{k-1}(-1)^i   \big[ k  \binom
{k-1}i \binom {r+i+1}{i+1}\\ \hskip2.5cm -   \binom{k-2}{i-1}\binom{r+i}{i+1}\big]
Q_{2,0}^{(k-1-i)p}Q_{2,1}^{i(p+1)},\\
\hskip 2.2cm\text{ if } j = kp+r \text{ with } 0\le k,r<p,\\
0, \hskip2cm \text{otherwise}.\end{cases}$$ 

{\rm(iv)} If $k+1<i<p $ and $0\le r<p$ then
$$St^{(i,kp+r)}Q_{2,1}=0.$$

{\rm(v)} If $0\le k <p$ then
$$St^{(k+1,kp)}Q_{2,1}= (-1)^kQ_{2,0}^{k+1}Q_{2,1}^{kp}.$$

{\rm(vi)}
$$St^{(1,j)}Q_{2,1}= \begin{cases} 
(k+1)Q_{2,0}^{r+1}\sum_{i=0}^k(-1)^i   \binom
{k}{i} \binom {r+i}iQ_{2,0}^{(k-i)p}Q_{2,1}^{i(p+1)}\\
\hskip2.3cm \text{if $j=kp+r$ with $0\le k,\ r<p$,}\\
0, \hskip2cm \text{otherwise}.\end{cases}$$
\end{thm}

\begin{proof} We prove  i) by induction on $k$. Let $k=0$. For $r=0$,   using Cartan formula we have
$St^{(i,0)}Q_{2,0}=0.$ For  $r>0$, applying the Cartan formula and the inductive hypothesis one gets
$$ St^{(i,r)}Q_{2,0}= Q_{2,0}St^{(i,r-1)}Q_{2,0}=0.$$
Part i) of the theorem holds for $k=0$. Suppose it is true for $0\le k<p-1$. For $i > k+1$,  we have 
\begin{align*} St^{(i,(k+1)p)}Q_{2,0}&=Q_{2,0}
St^{(i,kp+p-1)}Q_{2,0}
+(Q_{2,0}^p-Q_{2,1}^{p+1})St^{(i,kp)}Q_{2,0}\\
&\qquad -Q_{2,0}^{p+1}St^{(i,(k-1)p+p-1)}Q_{2,0}
-Q_{2,0}Q_{2,1}^pSt^{(i-1,kp)}Q_{2,0}\\
&=0.\end{align*}
Part i) of the theorem holds for $i>k+1$ and $r=0$.  Suppose it is true for $i > k+1$ and $0\le r<p-1$. Then using the inductive hypothesis we have
 \begin{align*} St^{(i,(k+1)p+r+1)}Q_{2,0}&=Q_{2,0}
St^{(i,(k+1)p+r)}Q_{2,0}
+(Q_{2,0}^p-Q_{2,1}^{p+1})St^{(i,kp+r+1)}Q_{2,0}\\
&\qquad -Q_{2,0}^{p+1}St^{(i,kp+r)}Q_{2,0}
-Q_{2,0}Q_{2,1}^pSt^{(i-1,kp+r+1)}Q_{2,0}\\
&=0.\end{align*}  
Part  i) of the theorem is proved.

 Now we prove Part  ii). Obviously, it is true for $i=0$. Suppose $i>0$  and Part ii) holds for $i-1$.  Using the Cartan formula and the inductive hypothesis, we have
\begin{align*} St^{(i,ip)}Q_{2,0} &= - Q_{2,0}Q_{2,1}^p  St^{(i-1,
(i-1)p)}Q_{2,0}\\  &=  -(-1)^{i-1}  Q_{2,0}Q_{2,1}^p Q_{2,0}^i
Q_{2,1}^{(i-1)p} \\
&=(-1)^{i}  Q_{2,0}^{i+1}Q_{2,1}^{ip}.\end{align*}

The proof of Part iii). It is well known that the Steenrod-Milnor operations are stable. That means $St^{(1,j)}x = 0$ for any homogeneous $x \in P_2$ with $\deg x < j$ (see M\`ui \cite{2}). Since $\deg Q_{2,0} = p^2-1 <p^2$, we have $St^{(1,j)}Q_{2,0} = 0$ for $j \geq p^2$. Suppose that $j < p^2$. Then using the $p$-adic expansion of $j$, we have
$$ j= kp + r \ \text{ for } 0 \le k, r <p.$$

Now we prove the theorem by double induction on $(k,r)$.

If $k=0$,  then using Part  i) of Theorem \ref{dl33}
we get $St^{(1,r)}Q_{2,0}=0$, so the theorem is true for $k=0$.
If $r=p-1$, then  $  \binom  {r+i+1}{i+1}=  \binom{r+i}{i+1}  =0$
in $\mathbb F_p$, so we obtain
$$St^{(1,kp+p-1)}Q_{2,0}=0.$$
If $r=0$, then the formula in the theorem becomes
$$St^{(1,kp)}Q_{2,0}=   -k
Q_{2,0}^2Q_{2,1}^p(Q_{2,0}^p-Q_{2,1}^{p+1})^{k-1}.$$
We prove this formula by induction on $k$. Suppose this formula holds for $0\le k<p-1$. For $r=0$, we have
\begin{align*} 
St^{(1,(k+1)p)}Q_{2,0}&=  (Q_{2,0}^p-Q_{2,1}^{p+1})
St^{(1,kp)}Q_{2,0}-Q_{2,0}Q_{2,1}^pSt^{(0,kp)}Q_{2,0}\\
&=-kQ_{2,0}^2Q_{2,1}^p(Q_{2,0}^p-Q_{2,1}^{p+1})^{k}
-Q_{2,0}^2Q_{2,1}^p(Q_{2,0}^p-Q_{2,1}^{p+1})^{k}\\
&=-(k+1)Q_{2,0}^2Q_{2,1}^p(Q_{2,0}^p-Q_{2,1}^{p+1})^{k}   .
\end{align*} 
So the formula holds for  $k+1$  and $r=0$. Suppose $0\le r<p-1$ and the formula holds for  $j=(k+1)p+r$. Applying the Cartan formula and the inductive hypothesis, one gets
\begin{align*} S&t^{(1,(k+1)p+r+1)}Q_{2,0}=
Q_{2,0}St^{(1,(k+1)p+r)}Q_{2,0}+ (Q_{2,0}^p-Q_{2,1}^{p+1})
St^{(1,kp+r+1)}Q_{2,0}\\   &\quad   -   Q_{2,0}^{p+1}
St^{(1,kp+r)}Q_{2,0} -Q_{2,0}Q_{2,1}^pSt^{(0,kp+r+1)}Q_{2,0}\\
\end{align*} 
\begin{align*} 
&=-Q_{2,0}^{r+3}Q_{2,1}^p\Big(k\sum_{i=0}^{k}(-1)^i \big[ \binom
{k}i \binom {r+i+1}{i+1}-   \binom{k-1}{i}\binom{r+i+2}{i+1}\\ &\quad + 
\binom {k-1}{i-1} \binom {r+i+1}{i}-  \binom{k-1}{i}\binom{r+i+1}{i+1}
\big] Q_{2,0}^{(k-i)p}Q_{2,1}^{i(p+1)}\\
&\quad+ \sum_{i=0}^{k}(-1)^i \binom ki \big[
\binom {r+i+1}{i+1}-\binom{r+i+1}{i}\big]
Q_{2,0}^{(k-i)p}Q_{2,1}^{i(p+1)}\\
& \quad-\sum_{i=0}^{k}(-1)^i \big[  \binom
{k-1}{i-1} \binom {r+i}{i+1}+\binom{k-2}{i-1}\binom{r+i+1}{i+1}\\ &\quad  +
\binom {k-2}{i-2} \binom {r+i}{i}-  \binom{k-2}{i-1}\binom{r+i}{i+1}
 \big] Q_{2,0}^{(k-i)p}Q_{2,1}^{i(p+1)}\Big)\\
&=-Q_{2,0}^{r+3}Q_{2,1}^p\sum_{i=0}^{k}(-1)^i \big[(k+1) \binom
{k}i \binom {r+i+2}{i+1}\\ &\hskip3cm-   \binom{k-1}{i-1}\binom{r+i+1}{i+1}
\big] Q_{2,0}^{(k-i)p}Q_{2,1}^{i(p+1)}.\end{align*}

Parts iv), v) and vi) are proved by the same argument as the previous one.

\end{proof}

\bigskip
This work was supported in part by a grant of Quy Nhon University Research Project.

{

\bigskip

Department of Mathematics, Quy Nhon University

170 An Duong Vuong, Quy Nhon, Binh Dinh, Viet Nam

\end{document}